\documentclass{amsart}

\usepackage{amssymb}

\newcommand{\hb}{{\boldsymbol{h}}}
\newcommand{\pb}{{\boldsymbol{p}}}
\newcommand{\CC}{{\mathbb{C}}}
\newcommand{\NN}{{\mathbb{N}}}
\newcommand{\RR}{{\mathbb{R}}}
\newcommand{\ZZ}{{\mathbb{Z}}}
\newcommand{\cI}{{\mathcal{I}}}
\newcommand{\cP}{{\mathcal{P}}}
\newcommand{\cQ}{{\mathcal{Q}}}
\newcommand{\Ideal}[1]{{\left\langle #1 \right\rangle}}
\newcommand{\Span}{{\mbox{\rm span}\,}}

\newtheorem{theorem}{Theorem}[section]
\newtheorem{lemma}[theorem]{Lemma}
\newtheorem{proposition}[theorem]{Proposition}
\newtheorem{corollary}[theorem]{Corollary}

\theoremstyle{definition}
\newtheorem{definition}[theorem]{Definition}
\newtheorem{example}[theorem]{Example}

\theoremstyle{remark}
\newtheorem{remark}[theorem]{Remark}

\numberwithin{equation}{section}

\begin{document}

\title{Kernels of discrete convolutions and subdivision operators}


\author{Tomas Sauer}
\address{Lehrstuhl f\"ur Mathematik mit Schwerpunkt Digitale
  Bildverarbeitung/FORWISS, University of Passau, Innstr.~43, D--94032
  Passau, Germany}
\curraddr{}
\email{Tomas.Sauer@uni-passau.de}
\thanks{}


\subjclass[2010]{Primary 39A14, 65D17, 65Q10.} 

\date{March 29, 2014}

\dedicatory{}

\commby{}

\begin{abstract}
  We consider kernels of discrete convolution operators or,
  equivalently, homogeneous solutions of partial difference operators
  and show that these solutions always have to be exponential
  polynomials. The respective polynomial space in connected directly
  though somewhat intricately to the multiplicity of the common zeros
  of certain multivariate polynomials, a concept introduced by
  Gr\"obner in the description of kernels of partial differential
  operators with constant coefficients. These results can are then used to
  determine the kernels of stationary subdivision operators as well.
\end{abstract}

\maketitle


\section{Introduction}
\label{sec:Intro}
This paper considers a simple question: which sequences $c :
\ZZ^s \to \RR$ can be kernels of convolution or subdivision
operators. Recall that a \emph{convolution operator} or \emph{filter}
based on a \emph{finite} impulse $h \in \ell_{00} (\ZZ^s)$ acts on as
sequence $c$ as 
\begin{equation}
  \label{eq:FilterDef}
  c = h * c = \sum_{\alpha \in \ZZ^s} h(\alpha) \, c( \cdot -
  \alpha), \qquad c \in \ell (\ZZ^s).
\end{equation}
Here and in what follows $\ell (\ZZ^s)$ denotes all multi-infinite
sequences, written as functions from $\ZZ^s \to \CC$ while $\ell_{00}
(\ZZ^s)$ stands for those with compact, i.e., finite, support: $\# \{
\alpha \in \ZZ^s \;:\; c(\alpha) \neq 0 \} < \infty$.

Convolution operators can also be viewed as \emph{partial difference
  operators}. Let $\tau_j : c \mapsto c(\cdot + \epsilon_j)$
denote the forward partial shift operator and $\epsilon_j$ the
$j$th unit index in $\NN_0^s$ as well as $\tau^\alpha :=
\tau_1^{\alpha_1} \cdots \tau_s^{\alpha_s}$, then
$$
h * c = \sum_{\alpha \in \ZZ^s} h(\alpha) \tau^{-\alpha} c =
h^* (\tau^{-1}) \, c,
$$
with the \emph{symbol}
$$
h^* (z) = \sum_{\alpha \in \ZZ^s} h(\alpha) z^\alpha, \qquad z \in
\CC_\times^s := ( \CC \setminus \{ 0 \} )^s,
$$
which associates to a finitely supported sequence $h \in \ell_{00}
(\ZZ^s)$ a \emph{Laurent polynomial}. Therefore, the kernels of the
convolution operators are the solution of the homogeneous difference
equation $h^* (\tau^{-1}) \, c = 0$.

It is not hard to guess what these solutions should be when taking
into account that for any
\emph{exponential sequence} $e_\theta : \alpha \mapsto \theta^\alpha$,
$\theta \in \CC_\times^s$, we get
\begin{equation}
  \label{eq:hetheta}
  h * e_\theta = \sum_{\alpha \in \ZZ^s} h(\alpha) \theta^{\cdot -
    \alpha} = e_\theta \, h^* (\theta^{-1}),  
\end{equation}
hence $e_\theta$ belongs to $\ker h$ if and only if $h^* (\theta^{-1})
= 0$. Therefore, the exponentials in the kernel of any
finitely supported convolution operator encoded in the zeros
of the symbol and it only remains to show that essentially no other
sequences can be annihilated by convolution operators. It is also to
be expected that the order of the zero at $\theta^{-1}$ will affect
the structure of the kernel and indeed, it will allow for some
exponential polynomial sequences.

The following classical result for $d=1$ is widely used
in systems theory and stated, for example, in
\cite[p.~543ff]{jordan65:_calcul} or, more as some type of ``cooking
recipe'', in \cite{goldberg58:_introd_differ_equat}.

\begin{theorem}\label{T:Kernel1D}
  Let $h \in \ell_{00} (\ZZ)$ whose symbol factors as
  $$
  h^* (z) = c z^m \prod_{\theta \in \Theta} ( z - \theta^{-1}
  )^{k_\theta}, \qquad k_\theta \in \NN,
  $$
  then
  $$
  \ker ( h * (\cdot) ) = \bigoplus_{\theta \in \Theta} e_\theta \, \Pi_{k_\theta-1}.
  $$
\end{theorem}

\noindent
Here, $\Pi_k$ denotes the vector space of all polynomials of degree at
most $k$, hence the \emph{multiplicity} of the zero at $\theta^{-1}$ directly
corresponds to the degree of the exponential polynomial space that
belongs to the kernel of the convolution operator. 

Our goal will be to give a complete analog of Theorem~\ref{T:Kernel1D} in
several variables, which, of course, will need a more careful
treatment of the (common) zeros of polynomials and in particular of their
multiplicities. Multiplicities of zeroes of polynomial ideals have
been considered for example in 
\cite{deBoorRon91,MarinariMoellerMora96}, but the main results are already
mentioned in \cite{GroebnerII}, where Gr\"obner refers to his
papers
\cite{groebner37:_ueber_macaul_system_bedeut_theor_differ_koeff,groebner39:_ueber_eigen_integ_differ_koeff},
where not only the concept of multiplicities is introduced and
clarified, but where he also solves the continuous counterpart of our
question, describing the kernels of partial differential operators.

Based on Gr\"obners multiplicity theory, we will state and prove the
counterpart of Theorem~\ref{T:Kernel1D} for zero dimensional ideals in
Section~\ref{sec:ConvKern}, while in Section~\ref{sec:SubdKern} we
briefly apply these results to also describe the kernels of stationary
subdivision operators in several variables. 

\section{Kernels of convolution operators}
\label{sec:ConvKern}
We begin by fixing some terminology. Let $\Pi = \RR [z] =
\RR[z_1,\dots,z_s]$ denote the ring of polynomials in $s$ variables
over $\RR$, and let $\deg f$ denote the \emph{total degree} of $f \in \Pi$. A
polynomial $f \in \Pi$ is called \emph{homogeneous} if it can be
written as 
$$
f (z) = \sum_{|\gamma| = \deg f} f_\gamma \, z^\gamma, \qquad z^\gamma
:= z_1^{\gamma_1} \cdots z_s^{\gamma_s},
$$
and we write $\Pi^0$ for all homogeneous polynomials, $\Pi_k$ for all
polynomials $f$ with $\deg f \le k$ and $\Pi_k^0$ for all homogeneous
$f$ with $\deg f = k$, $k \in \NN_0$. By $\Lambda (f) \in \Pi_{\deg
  f}^0$ we denote the homogeneous leading term of $f$, defined by $f -
\Lambda (f) \in \Pi_{\deg f - 1}$.
To a polynomial $q \in \Pi$ we
associated the constant coefficient partial difference operator
$$
q(D) = q \left( \frac{\partial}{\partial
    z_1},\dots,\frac{\partial}{\partial z_s} \right)
= \sum_{\alpha \in \ZZ^s} q_\alpha \frac{\partial^{|\alpha|}}{\partial
  z^\alpha}
= \sum_{\alpha \in \ZZ^s} q_\alpha D^\alpha,
$$
and call a subspace $\cP$ of $\Pi$ \emph{$D$--invariant} if $\Pi(D)
\cP = \cP$, that is, $q(D)p \in \cP$, $p \in \cP$, $q \in
\Pi$. Finally, we introduce an inner product $(\cdot,\cdot) : \Pi
\times \Pi \to \RR$ by setting
\begin{equation}
  \label{eq:InnerProdDef}
  (f,g) := ( f(D) g ) (0) = \sum_{\alpha \in \NN_0^s} \alpha! \, f_\alpha g_\alpha.
\end{equation}
This inner product was used in \cite{deBoorRon91} and also
in the construction of the \emph{least interpolant},
cf. \cite{deBoorRon92a}. I learned that it is sometimes called
``Bombieri inner product'' or ``Fisher inner product'' though
unfortunately I cannot provide references; also, Charles Dunkl
(private communication) mentioned that Calderon used this inner
product in the context harmonic polynomials. For our purposes here it will
turn out to be more useful than the ``canonical'' inner product $(f,g)
= \sum f_\alpha g_\alpha$ that gives rise to Macaulay's inverse
systems,
cf. \cite{groebner37:_ueber_macaul_system_bedeut_theor_differ_koeff,
  GroebnerII, Sauer01}. 

\subsection{$D$--invariant spaces}
\label{ssec:DInvar}
The identity
$$
( p(D)f,g ) = (f,pg), \qquad f,p,g \in \Pi
$$
is easily derived from (\ref{eq:InnerProdDef}) and directly yields the
following observation.

\begin{lemma}\label{L:InvarIdeal}
  A subspace $\cQ \subseteq \Pi$ is $D$--invariant if and only if
  $\cQ^\perp = \{ f \;:\; (\cQ,f) = 0 \}$ is an ideal.
\end{lemma}

\noindent
Based on Lemma~\ref{L:InvarIdeal} one can construct a homogeneous
basis for the $D$--invariant space $\cQ$ by successively constructing
bases for
$$
\cP_j = \left\{ f \in \Pi_j^0 \;:\; (f,\Lambda (\cQ^\perp) ) = 0
\right\}, \qquad j = 0,\dots,\deg \cQ := \max \{ \deg q \;:\; q \in
\cQ \},
$$
cf. \cite{Sauer01}. Since $\bigoplus_j \cP_j \equiv \Pi / \cQ^\perp$, it follows
that $\cP_0 + \cdots + \cP_{\deg \cQ} = \cQ$ and therefore $\cQ$ has a
homogeneous basis which will be denoted by $Q$. Since $(f,g) = 0$ if
$\deg f \neq \deg g$, we can moreover assume that $Q$ is a
\emph{orthonormal homogeneous basis}, that is,
\begin{equation}
  \label{eq:Qortho}
  (q,q') = \delta_{q,q'}, \qquad q,q' \in \cQ.
\end{equation}
Hence, any $f \in \cQ$ can be written as
\begin{equation}
  \label{eq:fOrthoExpansion}
  f = \sum_{q \in Q} (f,q) \, q = \sum_{q \in Q} ( q(D)f ) (0) \, q  
\end{equation}
from which we can conclude for $x,y \in \RR^s$ that
$$
f(x + y) = \sum_{q \in Q} (f(\cdot+y),q) \, q(x) = \sum_{q \in Q}
( q(D)f ) (y) \, q (x),
$$
hence, by symmetry,
\begin{equation}
  \label{eq:fx+yFormula}
  f(x + y) = \sum_{q \in Q} ( q(D)f ) (y) \, q (x)
  = \sum_{q \in Q} ( q(D)f ) (x) \, q (y).
\end{equation}
Note that (\ref{eq:fx+yFormula}) in particular implies that any
$D$--invariant space is shift invariant.



\subsection{Zero dimensional ideals}
\label{ssec:ZeroDim}
In several variables, a single convolution $h * c$ cannot be sufficient 
to have a finite dimensional kernel. Indeed, (\ref{eq:hetheta}) shows that
$h * e_\theta = 0$ for any zero $\theta^{-1}$ of $h^*$, which can be a
whole algebraic variety, hence
usually not even a countable set. Therefore, we emerge from a finite
set $H \subset \ell_{00} (\ZZ^s)$, consider the \emph{ideal}
$$
\Ideal{H^*} = \left\{ \sum_{h \in H} f_h \, h^* \;:\; f_h \in \Pi \right\}
$$
generated by $h^*$, $h \in H$, and request that the ideal is
\emph{zero dimensional}, that
is, there exists a \emph{finite} set $\Theta \subset \CC^s$ such that
$$
H^* (\Theta^{-1}) = 0, \qquad \mbox{i.e.,} \qquad h^* (\theta^{-1}) =
0,\quad h \in H, \theta \in \Theta.
$$
Since $h * c = 0$ implies $(g * h) * c = g*h*c = 0$ with $(g*h)^* =
g^* h^*$ for any finite
filter $g$, the kernel does not depend of the
generating set $H$ but of the ideal $\Ideal{H^*}$.

In Theorem~\ref{T:Kernel1D} we have seen that multiplicities of the
zeros play a fundamental role for the structure of the kernel. To
extend this to the case of several variable, we recall the
following classical description of the multiplicities of common zeroes
of ideals, see also \cite{deBoorRon91,MarinariMoellerMora96}.

\begin{theorem}[\cite{groebner37:_ueber_macaul_system_bedeut_theor_differ_koeff}]
  $\cI \subset \Pi$ is a zero dimensional ideal if and only if there
  exists a finite set $\mathrm{Z} \subset \CC^s$ and $D$--invariant
  subspaces $\cQ_\zeta$, $\zeta \in \mathrm{Z}$, such that
  $$
  f \in \cI \qquad \Leftrightarrow \qquad q(D) f (\zeta) = 0, \quad q
  \in \cQ_\zeta, \, \zeta \in \mathrm{Z}.
  $$
\end{theorem}

\noindent
In \cite{GroebnerII}, the dimension $\dim \cQ_\zeta$ of $\cQ_\zeta$ is
called the \emph{multiplicity} of the zero $\zeta$, but it will be
more appropriate here to work with the spaces $\cQ_\theta$ themselves.
The the dimension of
$\cQ_\zeta$ alone is not sufficient to fully describe the nature of
the zero is easily seen from the
the two examples
$$
\cQ_\zeta = \{ 1,x,y \}, \qquad \cQ_\zeta = \{ 1,x+y,(x+y)^2 \}
$$
of a triple zero in two variables.

\noindent
It is worthwhile to remark that generally the symbol $h^*$ is not a
polynomial but a Laurent polynomial, hence $h^* = (\cdot)^\alpha f$
for some $\alpha \in \ZZ^s$ and $f \in \Pi$. Since it is easily seen
that $\ker H$ is a shift invariant space, we can always shift the
impulse responses $h \in H$ such that $h^* \in \Pi$. However, one must
keep in mind that a ``spurious'' zero of $h^*$ at zero do not count
when considering $\ker H$; this is a well--known effect also in the
context of smoothness analysis of refinable functions, see
\cite{MoellerSauer04}.

\begin{definition}\label{D:ZeroDimFilt}
  A finite set $H \subset \ell_{00} (\ZZ^s)$ of impulse responses is
  called \emph{zero dimensional} if the ideal $\Ideal{H^*}$ is zero
  dimensional or, equivalently, if there exist a finite subset $\Theta
  \subset \CC_\times^s$ and finite dimensional $D$--invariant spaces
  $\cQ_\theta$, $\theta \in \Theta$, such that
  $$
  q(D) h^* (\theta^{-1}) = 0, \qquad q \in \cQ_\theta, \, \theta \in
  \Theta, \, h \in H.
  $$
\end{definition}

\subsection{Annihilation of exponential polynomials}
\label{ssec:Main}
In order to formulate the main results of this paper, we need some more
terminology. The \emph{partial difference operator} $\Delta^\alpha$,
acting on $\ell (\ZZ^s)$ is
recursively defined as
$$
\Delta^{\alpha + \epsilon_j} = (\tau^{\epsilon_j} - I ) \Delta^\alpha,
\qquad \alpha \in \NN_0^s, \quad j=1,\dots,s.
$$
We define an operator $L : \Pi \to \Pi$ as
\begin{equation}
  \label{eq:LOpDef}
  L f (x) = \sum_{|\gamma| \le \deg f} \frac{1}{\gamma!} \Delta^\gamma
  f (0) \, x^\gamma
\end{equation}
and note that $\Lambda (Lf) = \Lambda( f )$ as well as $\deg L f =
\deg f$. This immediately leads to the following observation.

\begin{lemma}\label{L:LLemma}
  $L$ is a degree preserving linear isomorphism $\Pi \to \Pi$ and
  $\Pi_k \to \Pi_k$ for any $k \in \NN$. In particular, there exists
  an inverse $L^{-1}$ on $\Pi$ as well as on $\Pi_k$, $k \in \NN_0$.
\end{lemma}

\noindent
Next, we introduce the \emph{scaling operator} $\sigma_\theta$, $\theta \in
\CC_\times^s$, as
$$
\sigma_\theta f (z) := f (\theta z) := f (\theta_1 z_1,\dots,\theta_s
z_s)
$$
with the abbreviation $\sigma_- :=
\sigma_{(-1,\dots,-1)}$. 
For $\theta \in \CC_\times^s$ and a $D$--invariant subspace $\cQ_\theta
\subset \Pi$, we define
$$
\widehat \cQ_\theta = \sigma_\theta \cQ_\theta = \left\{ \sigma_\theta
  q \;:\; q \in \cQ_\theta \right\}, 
$$
and note that $\widehat \cQ_\theta$ is also $D$--invariant since
$$
p(D) q(\theta \cdot) = \sum_{\alpha \in \NN_0^s} p_\alpha
\theta^\alpha ( D^\alpha q ) (\theta \cdot), \qquad p = \sum_{|\alpha| \le
  \deg p} p_\alpha \, (\cdot)^\alpha,
$$
and $D^\alpha q$ can be expanded in terms of $\cQ_\theta$. Moreover,
we introduce to $\cQ_\theta$ the space
\begin{equation}
  \label{eq:PthetaDef}
  \cP_\theta := \sigma_- L \widehat \cQ_\theta = \Span \{ L^{-1}
  \sigma_\theta q \;:\; q \in Q_\theta \},
\end{equation}
where again $Q_\theta$ denotes an homogeneous orthonormal basis of
$\cQ_\theta$.

\begin{example}\label{Ex:QPSpaces}
  For the $D$--invariant space $\cQ_\theta = \Span \{
  1,(x+y),(x+y)^2 \}$ and $\theta = (\theta_1,\theta_2)$ with
  $\theta_1 \neq \theta_2$ we get $\widehat \cQ_\theta = \Span \{ 1,
  \theta_1 x + \theta_2 y, (\theta_1 x + \theta_2 y)^2 \} \neq
  \cQ_\theta$. A straightforward computation yields
  \begin{align*}
    L 1 & = 1 \\
    L ( \theta_1 x + \theta_2 y ) & = \theta_1 x + \theta_2 y \\
    L ( \theta_1 x + \theta_2 y )^2 & = ( \theta_1 x + \theta_2 y )^2
    + \theta_1^2 x + \theta_2^2 y,
  \end{align*}
  which shows that
  $$
  \cP_\theta = \Span \left\{ 1, \theta_1 x + \theta_2 y, ( \theta_1 x +
    \theta_2 y )^2 - ( \theta_1^2 x + \theta_2^2 y ) \right\}
  $$
  is not spanned by
  homogeneous polynomials, hence cannot be $D$--invariant as soon as
  $\theta_1 \neq \theta_2$. Moreover, $\cP_\theta$ is not
  $\sigma_-$ invariant in that case.
\end{example}

\noindent
Nevertheless, $\cP_\theta$ has a fundamental invariance property.

\begin{lemma}\label{L:PthetaShift}
  The space $\cP_\theta$ is shift invariant.
\end{lemma}

\begin{proof}
  Any $p \in \cP_\theta$ can be written as
  $$
  p(x) = \sum_{|\gamma| \le \deg q} \frac{1}{\gamma!} \Delta^\gamma q
  (0) (-x)^\gamma
  = \sum_{|\gamma| \le \deg q} (-1)^{|\gamma|} \frac{1}{\gamma!}
  \Delta^\gamma q (0) x^\gamma
  $$
  for some $q \in \widehat \cQ_\theta$, hence, since $\Delta^\gamma q
  (0) = 0$ for $|\gamma| > \deg q$
  \begin{align*}
    p(x+y) & = \sum_{\gamma \in \ZZ^s} (-1)^{|\gamma|}
    \frac{1}{\gamma!} \Delta^\gamma q (0) \sum_{\beta \le \gamma}
    {\gamma \choose \beta} x^\beta y^{\gamma-\beta} \\
    & = \sum_{\gamma \in \ZZ^s} (-1)^{|\gamma|} \frac{1}{\gamma!} \sum_{\beta \le
      \gamma} {\gamma \choose \beta} \Delta^\beta
    \Delta^{\gamma-\beta} q (0) x^\beta y^{\gamma-\beta} \\
    & = \sum_{\alpha,\beta \in \ZZ^s} \frac{1}{(\alpha+\beta)!}
    {\alpha+\beta \choose \beta} \Delta^\beta \Delta^\alpha q (0)
    (-x)^\beta (-y)^\alpha \\ 
    & = \sum_{\beta \in \ZZ^s} \frac{(-x)^\beta}{\beta!} \Delta^\beta
    \sum_{\alpha \in \ZZ^s} \frac{1}{\alpha!} \Delta^\alpha q(0)
    (-y)^\alpha
    = \sum_{\beta \in \ZZ^s} \frac{(-x)^\beta}{\beta!} \Delta^\beta
    q_y (0) \\
    & = \sigma_- L q_y (x)
  \end{align*}
  and since
  $$
  q_y := \sum_{|\alpha| \le \deg q} \frac{(-y)^\alpha}{\alpha!} \Delta^\alpha q
  $$
  belongs to $\widehat \cQ_\theta$ as this space is $D$--invariant,
  we can conclude that $p(x+y) \in \cP_\theta$
  as well.
\end{proof}

\noindent
Recalling an argument from
\cite{groebner39:_ueber_eigen_integ_differ_koeff}, we note that the
shift invariance of $\cP_\theta$ implies that for any $f \in
\cP_\theta$ we have
$$
f (x+y) = \sum_{p \in P} g(y) \, p (x), \qquad P_\theta = \sigma_- L
\sigma_\theta Q,
$$
and since, by symmetry, also $g \in \cP_\theta$, we get that
\begin{equation}
  \label{eq:PthetaExpand}
  f(x+y) = \sum_{p,p' \in P} a_{p,p'} (f) \, p(x) p'(y), \qquad
  a_{p,p'} (f) = a_{p',p} (f) \in \RR.
\end{equation}
In particular, any basis element $p \in P_\theta$ can be written as
$$
p (x+y) = \sum_{p' \in P_\theta} g_{p,p'} (y) \, p' (x), \qquad
g_{p,p'} = \sum_{\widetilde p \in P_\theta} a_{\widetilde p,p'} (p) \,
\widetilde p,
$$
or, in matrix notation, $P_\theta (\cdot + y) = G(y) P_\theta$ with
$G(0) = I$; moreover, it was it was shown in
\cite{groebner39:_ueber_eigen_integ_differ_koeff} that $\det G(y) =
1$, $y \in \RR^s$.
After defining the \emph{unimodular} polynomial matrices
\begin{equation}
  \label{eq:gqqDef}
  \widetilde G := \left[ g_{q,q'} := g_{\sigma_- L \sigma_\theta q,
      \sigma_- L \sigma_\theta q'} \;:\; q,q' \in Q_\theta \right] \in
  \Pi^{Q_\theta \times Q_\theta}
\end{equation}
and
\begin{equation}
  \label{eq:gpqDef}
  \widehat G := \left[ g_{p,q} := g_{p, \sigma_- L \sigma_\theta
    q} \;:\; p \in P_\theta, \, q \in Q_\theta \right] \in \Pi^{P_\theta \times Q_\theta}
\end{equation}
which only differ in their way of indexing,
we have all tools at hand to prove the next result.

\begin{proposition}\label{P:MultAnnihil}
  Let $\theta \in \CC_\times^s$ and $\cQ_\theta$ be a finite
  dimensional $D$--invariant subspace of $\Pi$. Then the following
  statements are equivalent:
  \begin{enumerate}
  \item\label{it:PMultAnnihil1}
    $h * ( \cP_\theta e_\theta ) = 0$, where $\cP_\theta$ is defined
    in (\ref{eq:PthetaDef}).
  \item\label{it:PMultAnnihil2}
    $q(D) h^* (\theta^{-1}) = 0$, $q \in \cQ_\theta$.
  \end{enumerate}
\end{proposition}

\begin{proof}
  We first note that the Newton formula for the Lagrange interpolant,
  \cite{IsaacsonKeller66,Steffensen27},
  yields for any polynomial $f \in \Pi$ that
  \begin{equation}
    \label{eq:NewtonFormula}
    f = \sum_{|\gamma| \le \deg f} \frac{1}{\gamma!} \Delta^\gamma f(0) \,
    (\cdot)_\gamma, \qquad
    (x)_\gamma = \prod_{j=1}^s \prod_{k=0}^{\gamma_j-1} (x_j - k),
  \end{equation}
  hence, for $p \in \cP_\theta$, $p = \sigma_- p'$, $p' \in L^{-1}
  \widehat \cQ_\theta$, and $\alpha \in \ZZ^s$,
  \begin{align*}
    h * ( p e_\theta ) (\alpha) & = \sum_{\beta \in \ZZ^s} h(\beta)
    p'( \beta - \alpha ) \, \theta^{\alpha - \beta} \\
    & = \theta^\alpha \sum_{\beta \in \ZZ^s} h(\beta) \sum_{|\gamma| \le
      \deg p} \frac{1}{\gamma!} \Delta^\gamma (\tau^{-\alpha} p')(0)
    (\beta)_\gamma \theta^{-\beta} \\
    & = \theta^\alpha \sum_{|\gamma| \le \deg p} \frac{1}{\gamma!} \Delta^\gamma
    (\tau^{-\alpha} p')(0) \, \theta^{-\gamma} \, \left( D^\gamma \sum_{\beta \in \ZZ^s}
      h(\beta) (\cdot)^\beta \right) (\theta^{-1}) \\
    & = \theta^\alpha ( L \tau^{-\alpha} p' )( \theta^{-1} D ) h^* (\theta^{-1}).
  \end{align*}
  By (\ref{eq:PthetaExpand})
  it follows that
  \begin{align}
    \nonumber
    h * ( p e_\theta ) (\alpha) & = \theta^\alpha \sum_{p_1 \in
      P_\theta} \sum_{q \in Q_\theta} a_{p_1,\sigma_- L \sigma_\theta
      q} (p) p_1 (\alpha) ( L L^{-1} \sigma_\theta q )( \theta^{-1} D
    ) h^* (\theta^{-1}) \\
    \label{eq:hpethta}
    & = \theta^\alpha \sum_{q \in Q_\theta} \left( \sum_{p_1 \in
        P_\theta} a_{p_1,\sigma_- L \sigma_\theta
        q} (p) \, p_1 (\alpha) \right) \, q(D) h^* (\theta^{-1})
    \\
    \nonumber
    & = \theta^\alpha \sum_{q \in Q_\theta} p_q (\alpha) \,  q(D) h^*
    (\theta^{-1}),
  \end{align}
  with
  $$
  p_q := \sum_{p_1 \in
    P_\theta} a_{p_1,\sigma_- L \sigma_\theta
    q} (p) \, p_1 \in \cP_\theta.
  $$
  Consequently, (\ref{it:PMultAnnihil2}) implies
  (\ref{it:PMultAnnihil1}) while
  for the converse we only need to set $p := \sigma_- L^{-1}
  \sigma_\theta q$ for $q \in Q_\theta$ to get, according to
  (\ref{eq:gqqDef}),
  $$
  0 = h * ( p e_\theta ) = e_\theta \sum_{q' \in Q_\theta} g_{q,q'} \,
  q' (D) h^* (\theta^{-1}), \qquad q \in Q_\theta,
  $$
  which gives $0 = e_\theta \widetilde G \, Q(D) h^* (\theta^{-1})$ and since
  $\det \widetilde G \equiv 1$, we can conclude that
  (\ref{it:PMultAnnihil1}) implies (\ref{it:PMultAnnihil2}) as well.
\end{proof}

\begin{remark}
  Like in the univariate case, the local space to be annihilated is a
  exponential polynomial space $\cP_\theta e_\theta$, however, it is
  generally not the same as the multiplicity space $\cQ_\theta$, see
  Example~\ref{Ex:QPSpaces}.
\end{remark}

\subsection{Kernels of convolutions}
Now we have all the tools at hand to give the main result of this
paper.

\begin{theorem}\label{T:MainTheorem}
  If $H$ is a zero dimensional set of impulse responses with zero set
  $\Theta^{-1}$ and multiplicities $\cQ_\theta$, $\theta \in \Theta$,
  respectively, then 
  \begin{equation}
    \label{eq:KernelFormulaGen}
    \ker H = \bigoplus_{\theta \in \Theta} \cP_\theta \, e_\theta,
    \qquad \cP_\theta := \sigma_- L^{-1} \sigma_\theta \cQ_\theta.
  \end{equation}
\end{theorem}

\noindent
The proof of Theorem~\ref{T:MainTheorem} is split into the following
two propositions.

\begin{proposition}\label{P:MainTheorem1}
  With the assumptions of Theorem~\ref{T:MainTheorem} we have that
  \begin{equation}
    \label{eq:KernelFormulaGen1}
    \ker H \supseteq \bigoplus_{\theta \in \Theta} \cP_\theta \, e_\theta.
  \end{equation}
\end{proposition}

\begin{proof}
  We identify $p_\theta \in \cP_\theta$, $\theta \in \Theta$, with
  its coefficient vector with respect to the basis $P_\theta$, and write
  $$
  p_\theta = p_\theta^T P := \sum_{p \in P_\theta} p_{\theta,p} \, p.
  $$
  Expanding (\ref{eq:hpethta}) further, we obtain
  \begin{align*}
    \lefteqn{
      h * \sum_{\theta \in \theta} p_\theta e_\theta
      = \sum_{\theta \in \Theta} \sum_{p \in P_\theta} \sum_{q \in
        q_\theta} a_{p,\sigma_- L \sigma_\theta q} ( p_\theta ) p \,
      q(D) h^* (\theta^{-1} )
    } \\
    & = \sum_{\theta \in \Theta} \sum_{p,p' \in P_\theta} \sum_{q \in
      q_\theta} p_{\theta,p'} a_{p,\sigma_- L \sigma_\theta q} ( p'
    ) p \, q(D) h^* (\theta^{-1} )
    =  \sum_{\theta \in \Theta} e_\theta \, p_\theta^T \, \widehat G_\theta \,
    Q_\theta (D) h^* (\theta^{-1}).
  \end{align*}
  Since $Q_\theta (D)
  h^* (\theta^{-1}) = 0$ by assumption, (\ref{eq:KernelFormulaGen1}) follows.
\end{proof}

\begin{proposition}\label{P:MainTheorem2}
  With the assumptions of Theorem~\ref{T:MainTheorem} we have that
  \begin{equation}
    \label{eq:KernelFormulaGen2}
    \ker H \subseteq \bigoplus_{\theta \in \Theta} \cP_\theta \, e_\theta.
  \end{equation}
\end{proposition}

\begin{proof}
  We use induction on $\# \Theta$ where the case $\# \Theta = 1$ is
  covered by Proposition~\ref{P:MultAnnihil}.

  To advance the induction hypothesis, let $\Theta' = \Theta \cup
  \{\theta'\}$ and assume that the result has been proved for $\#
  \Theta$. With
  $$
  \cI_\Theta := \left\{ f \in \Pi \;:\; q(D) f (\theta^{-1}) = 0,\, q
    \in \cQ_\theta,\, \theta \in \Theta \right\}
  $$
  and any basis $H_\Theta^*$ of $\cI_\Theta$ we get the quotient ideal
  representation 
  \begin{equation}
    \label{eq:Quotideal}
    \Ideal{H_\Theta^*} = \cI_{\Theta} : \cI_{\{\theta'\}}  =
    \cI_{\Theta'} : \cI_{\{\theta'\}} = \Ideal{H_{\Theta'}^*} :
    \cI_{\{\theta'\}}
  \end{equation}
  as well as (\ref{eq:KernelFormulaGen1}). Let $H_{{\theta'}}^*$ be any
  basis of $\cI_{\{\theta'\}}$, then
  (\ref{eq:Quotideal}) can be rephrased as
  $$
  h * c \in \ker H_\Theta, \qquad h \in H_{\theta'}, \, c \in \ker H.
  $$
  Using the vector $\hb = \left[ h \;:\; h \in H_{\theta'} \right]$, this
  can be written by the induction hypothesis
  (\ref{eq:KernelFormulaGen1}) as
  \begin{equation}
    \label{eq:VecConvFun}
    \hb * c = \sum_{\theta \in \Theta} \pb_\theta \, e_\theta, \qquad
    \pb_\theta = \left[ p_{\theta,h} \;:\; h \in H_{\theta'} \right] \in
    \cP_\theta^{H_{\theta'}}.  
  \end{equation}
  We want to find $s_\theta \in \cP_\theta$, $\theta \in \Theta$, such that
  $$
  c = \sum_{\theta \in \Theta} s_\theta \, e_\theta
  $$
  satisfies (\ref{eq:VecConvFun}).
  To that end, we use a special basis $\hb$, consisting of the
  polynomials $f_{\theta,q} \in \cI_{\{\theta'\}}$. $q \in Q_\theta$,
  $\theta \in \Theta$ such that
  $$
  \widetilde q (D) f_{\theta,q} (\widetilde \theta^{-1}) =
  \delta_{q,\widetilde q}
  \delta_{\theta,\widetilde \theta}, \qquad \widetilde q \in
  Q_{\widetilde \theta}, \, \widetilde \theta \in \Theta,
  $$
  and a basis $H_{\Theta'}$ of $\cI_{\Theta'}$.
  The polynomials $f_{\theta,q}$ exist since the associated Hermite interpolation
  problem is an ideal one, cf. \cite{boor05:_ideal,Sauer06a}, and they
  are fundamental solutions for this problem. Since
  any element of $\cI_{\{\theta'\}}$ can be expressed as the sum of
  the Hermite interpolant and an element from $\cI_{\Theta'}$, this is
  a proper basis for the ideal $\cI_{\{\theta'\}}$. Let us again write
  $$
  \pb_\theta = \sum_{p \in P_\theta} \pb_{\theta,p} \, p \qquad
  \mbox{and} \qquad s_\theta = \sum_{p \in P_\theta} s_{\theta,p} \,
  p, \qquad \pb_{\theta,p} \in \RR^{H_{\theta'}}, \quad s_{\theta,p}
  \in \RR,
  $$
  as well as $s_\theta^T = [ s_{\theta,p} : p \in P_\theta ]$ and
  $p_{\theta,h}^T = [ p_{\theta,h,p} : p \in P_\theta ]$,
  respectively, for the row vectors of the coefficients,
  then the same computation as in the proof of
  Proposition~\ref{P:MainTheorem2} yields for $h \in H_\theta$
  $$
  h * \sum_{\theta \in \theta} s_\theta e_\theta
  =  \sum_{\theta \in \Theta} e_\theta s_\theta^T \, \widehat G_\theta \,
  Q_\theta (D) h^* (\theta^{-1}),
  $$
  and also gives by symmetry, for $h,h' \in H_{\theta'}$,
  \begin{align*}
    \lefteqn{
      h' * h * c = h' * \sum_{\theta \in \Theta} p_{\theta,h} e_\theta
    }\\
    & = \sum_{\theta \in \Theta} e_\theta p_{\theta,h}^T \widehat G_\theta \,
    Q_\theta (D) {h'}^* (\theta^{-1})
    = \sum_{\theta \in \Theta} e_\theta p_{\theta,h'}^T \widehat G_\theta \,
    Q_\theta (D) {h}^* (\theta^{-1}),
  \end{align*}
  that is,
  \begin{equation}
    \label{eq:hh'Ident}
    p_{\theta,h}^T \widehat G_\theta \,
    Q_\theta (D) {h'}^* (\theta^{-1}) = p_{\theta,h'}^T \widehat G_\theta \,
    Q_\theta (D) {h}^* (\theta^{-1}),
  \end{equation}
  which immediately gives that $p_{\theta,h} = 0$ whenever $Q_\theta
  (D) h^* (\theta^{-1}) = 0$, while for $h^* = f_{\theta,q}$ and ${h'}^* =
  f_{\theta,q'}$ we get
  \begin{equation}
    \label{eq:hh'Ident2}
    p_{\theta,h}^T \widehat G_\theta \delta_{q'} = p_{\theta,h'}^T \widehat
    G_\theta \delta_q, \qquad \delta_q := \left[ \delta_{q,q'} \;:\;
      q' \in Q_\theta \right].
  \end{equation}
  Setting $h^* = f_{\theta,q}$ and $P_\theta = \left[ \sigma_- L
    \sigma_\theta q \;:\; q \in Q_\theta \right]$, the requirement
  that the $s_\theta$ satisfy (\ref{eq:VecConvFun}) can be expressed
  by (\ref{eq:hh'Ident2}) as
  \begin{align*}
    s_\theta^T \widehat G \delta_q & = p_{\theta,f_{\theta,q}} = p_{\theta,h}^T
    \widehat G \, P_\theta (0)
    = \sum_{q' \in Q_\theta} (\sigma_- L
    \sigma_\theta q') (0) \, p_{\theta,f_{\theta,q}}^T \widehat G \delta_{q'} \\
    & = \sum_{q' \in Q_\theta} (\sigma_- L
    \sigma_\theta q') (0) \, p_{\theta,f_{\theta,q'}}^T \widehat G \delta_{q},
  \end{align*}
  from which it follows that
  \begin{equation}
    \label{eq:sthetafinal}
    s_\theta = \sum_{q' \in Q_\theta} (\sigma_- L \sigma_\theta q')
    (0) \, p_{\theta,f_{\theta,q'}}^T 
  \end{equation}
  guarantees is a solution for (\ref{eq:VecConvFun}). Since
  any two solutions $c,c'$ of (\ref{eq:VecConvFun}) must satisfy
  $\hb * (c-c') = 0$, it follows that
  $$
  c - c' \in \ker H_{\theta'}, \qquad {i.e.} \qquad
  c - c' = s_{\theta'} \, e_{\theta'},
  $$
  again by Proposition~\ref{P:MultAnnihil}. In other words, $c \in \ker H$
  implies that 
  $$
  c = \sum_{\theta \in \Theta'} s_{\theta'} \, e_{\theta'}, \qquad
  s_\theta \in \cP_\theta, \, \theta \in \Theta',
  $$
  which advances the induction hypothesis and completes the proof.
\end{proof}

\noindent
Theorem~\ref{T:MainTheorem} is the direct generalization of
Theorem~\ref{T:Kernel1D} to the case of several variables. The main
difference is that the $D$--invariant space $\cQ_\theta$ that
describes the multiplicity of the common zeros of the symbol is
mapped to the shift invariant space $\cP_\theta$ that describes which
polynomials to multiply to the exponential $e_\theta$. As
Example~\ref{Ex:QPSpaces} shows, these spaces need not coincide at
all, though they have same dimension, hence the same \emph{scalar}
multiplicity.
Nevertheless, the kernel space depends directly on the zeros and
their multiplicity and the bases of the two spaces can even be chosen
in such a way that they have they same homogeneous leading forms.

There is, however, an important special case, namely, when the $\cQ_\theta$ are
spanned by monomials, more precisely, a lower set of monomials:
$$
(\cdot)^\alpha \in \cQ_\theta \qquad \Rightarrow \qquad
(\cdot)^\beta \in \cQ_\theta, \quad \beta \le \alpha.
$$
In this case, $\widehat \cQ_\theta = \cQ_\theta$ and $L \cQ_\theta =
\cQ_\theta$, hence $\cP_\theta = \cQ_\theta$. This holds true in
particular for the case of \emph{zeros of order $k$} or \emph{fat
  points}, which is defined as 
$\cQ_\theta = \Pi_{k_\theta}$, $k_\theta \in \NN_0$. Since in one
variable multiplicities are always fat points, the discrepancy between
$\cQ_\theta$ and $\cP_\theta$ is indeed a truly multivariate phenomenon.

\subsection{Eigenvectors of convolutions}
\label{ssec:EigenVec}
A simple application of Theorem~\ref{T:MainTheorem} is to find
\emph{eigensequences} of convolution operators. Suppose that $H
\subset \ell_{00} (\ZZ^s)$ is again a finite set of impulse responses
and assume that there exist $\lambda_h \in \CC$ and $\alpha_h \in
\NN_0^s$ such that
\begin{equation}
  \label{eq:EigenProp}
  h * c = \lambda_h \, c (\cdot + \alpha_h), \qquad h \in H.
\end{equation}
This is equivalent to
$$
( h (\cdot + \alpha_h) - \lambda_h \delta ) * c = 0, \qquad h
\in H
$$
and thus depends on the zeros of the (Laurent) ideal
$$
\Ideal{ z^{-\alpha_h} h^* (z) - \lambda_h ) \;:\; h \in H} = \Ideal{
  h^* (z) - \lambda_h z^{\alpha_h} \;:\; h \in H}.
$$
Hence, also the eigensequences of convolution operators can be only
exponential polynomials.

\begin{corollary}
  If $\Ideal{h^* (z) - \lambda z^{\alpha_h} \;:\; h \in H}$ is zero
  dimensional with zeros $\Theta^{-1} \subset \CC_\times^s$ and
  respective multiplicities $\cQ_\theta$, then the solutions of
  (\ref{eq:EigenProp}) are $\bigotimes_\theta \cP_\theta e_\theta$ and
  the conditions on $h^*$ are
  $$
  q(D) h^* ( \theta^{-1} ) = \lambda_h ( q(D) (\cdot)^{\alpha_h})
  (\theta^{-1}), \qquad q \in \cQ_\theta, \, \theta \in \Theta, \, h
  \in H.
  $$
\end{corollary}

\noindent
The situation is particularly simple if no shifts are involved as then
$q(D) ( h^* - \lambda_h ) (\theta^{-1})$ yields the conditions
\begin{equation}
  \label{eq:SimpleEigen}
  h^* (\theta^{-1}) = \lambda_h, \qquad q(D) h^* (\theta^{-1}) = 0,
  \quad q \in \cQ_\theta,\, \deg q > 0.
\end{equation}

\section{Kernels of subdivision operators}
\label{sec:SubdKern}
As a final application of Theorem~\ref{T:MainTheorem} we have a brief
look at the kernels of \emph{subdivision operators} in several
variables. To that end, let $\Xi \in \ZZ^{s \times s}$ be a
\emph{expanding matrix} which means that all eigenvalues of $\Xi$ are
larger than one in modulus, or, equivalently, that $\| \Xi^{-k} \| \to
0$ as $k \to \infty$. A \emph{stationary subdivision operator} $S_a$ with
scaling matrix $\Xi$ and finitely supported \emph{mask} $a \in \ell_{00}
(\ZZ^s)$ acts on $\ell (\ZZ^s)$ in the convolution--like way
\begin{equation}
  \label{eq:SubdOpDiv}
  S_a c = \sum_{\alpha \in \ZZ^s} a (\cdot - \Xi \alpha) \, c(\alpha),
  \qquad c \in \ell (\ZZ^s).
\end{equation}
To analyze the kernels of such operators, we need a little bit more
terminology. By $E_\Xi := \Xi [0,1)^s \cap \ZZ^s$ we denote the set of
coset representers for $\ZZ^s / \Xi \ZZ^s$, i.e.,
$$
\ZZ^s = \bigcup_{\xi \in E_\Xi} \xi + \Xi \ZZ^s.
$$
Similarly, $E_\Xi' := \Xi^T [0,1)^s \cap \ZZ^s$ stands for the
representers of $\ZZ^s / \Xi^T \ZZ^s$.

An important tool will be the \emph{subsymbols} of $a$, defined as
$$
a_\xi^* (z) = \sum_{\alpha \in \ZZ^s} a(\xi + \Xi \alpha) z^\alpha,
\qquad \xi \in E_\Xi, \, z \in \CC_\times^s.
$$
It is easily seen that symbol and subsymbols are related via 
\begin{equation}
  \label{eq:SymbSubSymb}
  a^* (z) = \sum_{\xi \in E_\Xi} z^\xi \, a_\xi^* (z^\Xi)
\end{equation}
and
\begin{equation}
  \label{eq:SymbSubSymb2}
  a_\xi^* (z^\Xi) = \frac{1}{|\det \Xi|} \sum_{\xi' \in E_\Xi'}
  e^{-2\pi i \xi^T \Xi^{-T} \xi'} a^* ( e^{-2\pi i \Xi^{-T} \xi'} z ),
\end{equation}
cf. \cite{Sauer11:_shear_multir_multip_refin}. Here, $z^\Xi = (
z^{\xi_1},\dots,z^{\xi_s} )$, where the $\xi_j$ are the columns of
$\Xi$, i.e., $\Xi = [ \xi_1 \dots \xi_s ]$.
Splitting the requirement $S_a c = 0$ modulo $\Xi$, we get 
$$
0 = S_a c (\xi + \Xi \alpha) = \sum_{\beta \in \ZZ^s} a( \xi + \Xi
(\alpha - \beta) c (\beta) = a_\xi * c, \qquad \xi \in \Xi,
$$
from which the following conclusion can be drawn.

\begin{corollary}\label{C:CommonSubdiv}
  Suppose that the ideal $\Ideal{a_\xi^* \;:\; \xi \in E_\Xi}$ is zero
  dimensional with zeros $\Theta^{-1}$ and respective multiplicities
  $\cQ_\theta$. Then $\ker S_a = \bigoplus_{\theta \in \Theta}
  \cP_\theta e_\theta$.
\end{corollary}

\noindent
To describe the kernel of a subdivision scheme in terms of the symbol
$a^*$ alone, we say that $\zeta \in \CC_\times^s$ is a \emph{symmetric
zero} of $a^*$ if
\begin{equation}
  \label{eq:SymmZeroDef}
  a^* ( e^{-2\pi i \Xi^{-T} \xi'} \zeta ) = 0, \qquad \xi' \in
  E_\Xi'.
\end{equation}
Symmetric zeros of $a^*$ are in one-to-one correspondence with common
zeros of $a_\xi^*$.

\begin{lemma}\label{L:CommonSymm}
  $\zeta$ is a symmetric zero of $a^*$ if and only if $\zeta^\Xi$ is a
  common zero of $a_\xi^*$, $\xi \in \Xi$.
\end{lemma}

\begin{proof}
  The key to the proof are (\ref{eq:SymbSubSymb}),
  (\ref{eq:SymbSubSymb2}) and the simple observation that 
  \begin{equation}
    \label{eq:UnitRoots}
    (e^{-2\pi i \Xi^{-T} \xi})^\Xi = e^{-2\pi i \Xi^T \Xi^T \xi} =
    e^{-2\pi i \xi} = (1,\dots,1).  
  \end{equation}
  Indeed, if $\zeta$ is a symmetric zero, then (\ref{eq:SymbSubSymb2})
  immediately yields that $a_\xi^* (\zeta^\Xi) = 0$, $\xi \in
  \Xi$, while for the converse we use (\ref{eq:SymbSubSymb}) and
  (\ref{eq:UnitRoots}) to verify that
  $$
  a^* ( e^{-2\pi i \Xi^{-T} \xi'} \zeta ) = \sum_{\xi \in E_\Xi}
  e^{-2\pi i \xi^T \Xi^{-T} \xi'} a_\xi^* ( \zeta^\Xi ) = 0
  $$
  holds for $\xi' \in E_\Xi'$, hence $\zeta$ is a symmetric zero.
\end{proof}

\noindent
Therefore, we can describe the kernel of a subdivision operator in terms
of its symmetric zeros.

\begin{corollary}\label{C:SubdKernel}
  There exists a polynomial space $\cP_\theta$ with $\cP_\theta
  e_\theta \subseteq \ker S_a$ if and only if $\theta^{-\Xi^{-1}}$ is
  a symmetric zero of $a^*$.
\end{corollary}

\noindent
This result can be extended to zeros with multiplicity provided that
the structure of the multiplicity is simple enough. The following
corollary can be understood as a characterization of vanishing moments
of the associated synthesis filterbank,
cf. \cite{Sauer11:_shear_multir_multip_refin}.

\begin{corollary}\label{C:SubdKernelMult}
  For a subdivision operator $S_a$ with mask $a \in \ell_{00} (\ZZ^s)$
  and $\Theta \subset \CC_\times^s$ the following statements are equivalent:
  \begin{enumerate}
  \item $\theta^{\Xi^{-1}}$ is a symmetric zero of $a^*$ of order
    $k_\theta$, $\theta \in \Theta$.
  \item One has
    $$
    \ker S_a = \bigoplus_{\theta \in \Theta} \Pi_{k_\theta} \, e_\theta.
    $$
  \end{enumerate}
\end{corollary}

\begin{proof}
  The only thing left to prove is the issue of multiplicity. To that end,
  we note that
  $$
  \nabla a^* \left( (\cdot)^\Xi \right) = A_{\Xi}^* (z)
  \, ( \nabla a^* ) \left( (\cdot)^\Xi \right), 
  $$
  where
  $$
  A_{\Xi}^* (z) := 
  \left[
    \begin{array}{ccc}
      z_1^{-1} \\ & \ddots \\ & & z_s^{-1}
    \end{array}
  \right] \Xi \left[
    \begin{array}{ccc}
      z^{\xi_1} \\ & \ddots \\ & & z^{\xi_s}
    \end{array}
  \right]
  $$
  is nonsingular for $z \in \CC_\times^s$. Turning to the total
  derivatives $\nabla^j = \left[ \frac{\partial^j}{\partial z^\alpha}
    \;:\; | \alpha| = j \right]$ of order $j \le k$, we observe that
  that
  \begin{equation}
    \label{eq:nablajform}
    \nabla^j a^* \left( (\cdot)^\Xi \right) = \sum_{\ell = 0}^j
    A_{j,\ell,\Xi}^* (z) ( \nabla^\ell a^* ) \left( (\cdot)^\Xi \right),  
  \end{equation}
  where
  $$
  A_{j,j,\Xi}^* (z) = A_{\Xi}^* (z) \otimes \cdots \otimes A_{\Xi}^*
  (z)
  $$
  is the $j$--fold Kronecker product of $A_\Xi^*$ with itself and thus
  nonsingular for any $z \in \CC_\times^s$. This follows from applying
  $\nabla$ to (\ref{eq:nablajform}) which yields inductively
  \begin{align*}
    \lefteqn{
      \nabla^{j+1} a^* \left( (\cdot)^\Xi \right)  = \sum_{\ell = 0}^j
      \nabla A_{j,\ell,\Xi}^* (z) \, ( \nabla^\ell a^* ) \left(
        (\cdot)^\Xi \right) + 
      A_{j,\ell,\Xi}^* (z) \, \nabla ( \nabla^\ell a^* ) \left(
        (\cdot)^\Xi \right)} \\
    & = A_{j,j,\Xi}^* (z) \left[ A_\Xi^* (z) \,
      \nabla \frac{\partial^j}{\partial z^\alpha a^*} \;:\; |\alpha| =
      j \right] \left(
      (\cdot)^\Xi \right) + \sum_{\ell = 0}^j A_{j+1,\ell,\Xi}^* (z) \,
    ( \nabla^\ell a^* ) \left((\cdot)^\Xi \right) \\ 
    & = \left( A_{j,j,\Xi}^* (z) \otimes A_\Xi^* (z) \right)(
    \nabla^{j+1} a^* ) \left((\cdot)^\Xi \right) + \sum_{\ell = 0}^j
    A_{j+1,\ell,\Xi}^* (z) \, ( \nabla^\ell a^* ) \left((\cdot)^\Xi
    \right).
  \end{align*}
  Consequently, we have for any $\theta \in \CC_\times^s$ that
  $\left( \nabla^j a^* \left( (\cdot)^\Xi \right) \right)
  (\theta^{-1}) = 0$, $j=0,\dots,k_\theta$ if and only if $( \nabla^j
  a^* ) \left( \theta^{-\Xi} \right) = 0$, $j=0,\dots,k$.
  
  With this observation, the claim follows immediately from
  differentiating (\ref{eq:SymbSubSymb}) and (\ref{eq:SymbSubSymb2}).
\end{proof}


\bibliographystyle{amsplain}
\bibliography{../bibls/subdiv,../bibls/sigproc,../bibls/books,../bibls/wavelets,../bibls/interpol,../bibls/numerik,../bibls/compalg,../bibls/dipldiss,../bibls/approx,../bibls/cagd,../bibls/ideal}

\providecommand{\bysame}{\leavevmode\hbox to3em{\hrulefill}\thinspace}
\providecommand{\MR}{\relax\ifhmode\unskip\space\fi MR }
\providecommand{\MRhref}[2]{%
  \href{http://www.ams.org/mathscinet-getitem?mr=#1}{#2}
}
\providecommand{\href}[2]{#2}
\begin{thebibliography}{10}

\bibitem{boor05:_ideal}
C.~de Boor, \emph{Ideal interpolation}, Approximation Theory XI, Gaitlinburg
  2004 (C.~K. Chui, M.~Neamtu, and L.~L. Schumaker, eds.), Nashboro Press,
  2005, pp.~59--91.

\bibitem{deBoorRon91}
C.~de Boor and A.~Ron, \emph{On polynomial ideals of finite codimension with
  applications to box spline theory}, J. Math. Anal. and Appl. \textbf{158}
  (1991), 168--193.

\bibitem{deBoorRon92a}
\bysame, \emph{The least solution for the polynomial interpolation problem},
  Math. Z. \textbf{210} (1992), 347--378.

\bibitem{goldberg58:_introd_differ_equat}
S.~Goldberg, \emph{Introduction to difference equations}, John Wiley \& Sons,
  1958, Dover reprint 1986.

\bibitem{groebner37:_ueber_macaul_system_bedeut_theor_differ_koeff}
W.~Gr{\"o}bner, \emph{{\"U}ber das {M}acaulaysche inverse {S}ystem und dessen
  {B}edeutung f{\"u}r die {T}heorie der linearen {D}ifferentialgleichungen mit
  konstanten {K}oeffizienten}, Abh. Math. Sem. Hamburg \textbf{12} (1937),
  127--132.

\bibitem{groebner39:_ueber_eigen_integ_differ_koeff}
\bysame, \emph{{\"U}ber die algebraischen {E}igenschaften der {I}ntegrale von
  linearen {D}ifferentialgleichungen mit konstanten {K}oeffizienten}, Monatsh.
  Math. \textbf{47} (1939), 247--284.

\bibitem{GroebnerII}
\bysame, \emph{Algebraische {G}eometrie {II}}, B.I--Hochschultaschenb{\"u}cher,
  no. 737, Bibliographisches Institut Mannheim, 1970.

\bibitem{IsaacsonKeller66}
E.~Isaacson and H.~B. Keller, \emph{Analysis of {N}umerical {M}ethods}, John
  Wiley \& Sons, 1966.

\bibitem{jordan65:_calcul}
Ch. Jordan, \emph{Calculus of finite differences}, 3rd ed., Chelsea, 1965.

\bibitem{MarinariMoellerMora96}
M.~G. Marinari, H.~M. M{\"o}ller, and T.~Mora, \emph{On multiplicities in
  polynomial system solving}, Trans. Amer. Math. Soc. \textbf{348} (1996),
  no.~8, 3283--3321.

\bibitem{MoellerSauer04}
H.~M. M{\"o}ller and T.~Sauer, \emph{Multivariate refinable functions of high
  approximation order via quotient ideals of {L}aurent polynomials}, Adv.
  Comput. Math. \textbf{20} (2004), 205--228.

\bibitem{Sauer01}
T.~Sauer, \emph{Gr{\"o}bner bases, {H}--bases and interpolation}, Trans. Amer.
  Math. Soc. \textbf{353} (2001), 2293--2308.

\bibitem{Sauer06a}
\bysame, \emph{Polynomial interpolation in several variables: Lattices,
  differences, and ideals}, Multivariate Approximation and Interpolation
  (M.~Buhmann, W.~Hausmann, K.~Jetter, W.~Schaback, and J.~St\"ockler, eds.),
  Elsevier, 2006, pp.~189--228.

\bibitem{Sauer11:_shear_multir_multip_refin}
\bysame, \emph{Shearlet multiresolution and multiple refinement}, Shearlets
  (G.~Kutyniok and Labate D., eds.), Springer, 2011.

\bibitem{Steffensen27}
I.~F. Steffensen, \emph{Interpolation}, Chelsea Pub., New York, 1927.

\end{thebibliography}
\end{document}